\documentclass{amsart}

\usepackage{amsmath,amsthm,amsfonts,amscd,amssymb}

\newtheorem{thm}{Theorem}[section]
\newtheorem{cor}[thm]{Corollary}
\newtheorem{lem}[thm]{Lemma}

\theoremstyle{definition}

\theoremstyle{remark}
\newtheorem{rem}{Remark}[section]

\begin{document}

\title{Algebraic points of small height missing a union of varieties}
\author{Lenny Fukshansky}

\address{Department of Mathematics, 850 Columbia Avenue, Claremont McKenna College, Claremont, CA 91711}
\email{lenny@cmc.edu}

\begin{abstract}
Let $K$ be a number field, $\overline{\mathbb Q}$, or the field of rational functions on a smooth projective curve over a perfect field, and let $V$ be a subspace of $K^N$, $N \geq 2$. Let $Z_K$ be a union of varieties defined over $K$ such that $V \nsubseteq Z_K$. We prove the existence of a point of small height in $V \setminus Z_K$, providing an explicit upper bound on the height of such a point in terms of the height of $V$ and the degree of a hypersurface containing $Z_K$, where dependence on both is optimal. This generalizes and improves upon the results of  \cite{me:classical} and \cite{me:number}. As a part of our argument, we provide a basic extension of the function field version of Siegel's lemma \cite{thunder} to an inequality with inhomogeneous heights. As a corollary of the method, we derive an explicit lower bound for the number of algebraic integers of bounded height in a fixed number field.
\end{abstract}

\maketitle
\tableofcontents

\def\A{{\mathcal A}}
\def\AA{{\mathfrak A}}
\def\B{{\mathcal B}}
\def\C{{\mathcal C}}
\def\D{{\mathcal D}}
\def\E{{\mathcal E}}
\def\F{{\mathcal F}}
\def\x{{\mathcal H}}
\def\I{{\mathcal I}}
\def\J{{\mathcal J}}
\def\K{{\mathcal K}}
\def\kk{{\mathfrak K}}
\def\L{{\mathcal L}}
\def\M{{\mathcal M}}
\def\mm{{\mathfrak m}}
\def\MM{{\mathfrak M}}
\def\OO{{\mathfrak O}}
\def\R{{\mathcal R}}
\def\s{{\mathcal S}}
\def\V{{\mathcal V}}
\def\X{{\mathcal X}}
\def\Y{{\mathcal Y}}
\def\Z{{\mathcal Z}}
\def\H{{\mathcal H}}
\def\cee{{\mathbb C}}
\def\pee{{\mathbb P}}
\def\que{{\mathbb Q}}
\def\real{{\mathbb R}}
\def\zed{{\mathbb Z}}
\def\aaa{{\mathbb A}}
\def\ff{{\mathbb F}}
\def\kk{{\mathfrak K}}
\def\qbar{{\overline{\mathbb Q}}}
\def\kbar{{\overline{K}}}
\def\ybar{{\overline{Y}}}
\def\kkbar{{\overline{\mathfrak K}}}
\def\ubar{{\overline{U}}}
\def\eps{{\varepsilon}}
\def\ahat{{\hat \alpha}}
\def\bhat{{\hat \beta}}
\def\gt{{\tilde \gamma}}
\def\h{{\tfrac12}}
\def\dd{{\partial}}
\def\be{{\boldsymbol e}}
\def\bei{{\boldsymbol e_i}}
\def\bc{{\boldsymbol c}}
\def\bm{{\boldsymbol m}}
\def\bk{{\boldsymbol k}}
\def\bi{{\boldsymbol i}}
\def\bl{{\boldsymbol l}}
\def\bq{{\boldsymbol q}}
\def\bu{{\boldsymbol u}}
\def\bt{{\boldsymbol t}}
\def\bs{{\boldsymbol s}}
\def\bv{{\boldsymbol v}}
\def\bw{{\boldsymbol w}}
\def\bx{{\boldsymbol x}}
\def\bX{{\boldsymbol X}}
\def\bz{{\boldsymbol z}}
\def\bwy{{\boldsymbol y}}
\def\bY{{\boldsymbol Y}}
\def\bL{{\boldsymbol L}}
\def\ba{{\boldsymbol\alpha}}
\def\bb{{\boldsymbol\beta}}
\def\bet{{\boldsymbol\eta}}
\def\bxi{{\boldsymbol\xi}}
\def\bo{{\boldsymbol 0}}
\def\bol{{\boldkey 1}_L}
\def\ep{\varepsilon}
\def\p{\boldsymbol\varphi}
\def\q{\boldsymbol\psi}
\def\rank{\operatorname{rank}}
\def\aut{\operatorname{Aut}}
\def\lcm{\operatorname{lcm}}
\def\sgn{\operatorname{sgn}}
\def\spn{\operatorname{span}}
\def\md{\operatorname{mod}}
\def\Norm{\operatorname{Norm}}
\def\dim{\operatorname{dim}}
\def\det{\operatorname{det}}
\def\Vol{\operatorname{Vol}}
\def\rk{\operatorname{rk}}
\def\ord{\operatorname{ord}}
\def\ker{\operatorname{ker}}
\def\div{\operatorname{div}}
\def\Gal{\operatorname{Gal}}

\section{Introduction}
\label{intro}

A fundamental problem of Diophantine geometry is determining the existence of rational points in an algebraic variety over a fixed field. One standard approach to this problem is through {\it search bounds} on the {\it height} of points in question. Specifically, suppose we were able to prove that whenever there exists a point with coordinates in a field $K$, lying in some subset $U$ of our variety, then there must exist such a point in $U$ with height bounded above by an explicit constant $B$. A key property of height functions over many fields $K$ of arithmetic interest is that the set of all points over $K$ with height $\leq B$ is finite. Then, in order to determine if $U$ in fact contains a rational point over $K$, it is sufficient to test only a finite set of points. A classical instance of such an approach is the celebrated Siegel's lemma, which originated in the work of Thue \cite{thue} and Siegel \cite{siegel} in transcendental number theory. In this paper we consider the problem of finding points of small height in a vector space outside of a union of a finite collection of varieties, which can be viewed as an extension of Siegel's lemma. This generalizes previous results of the author~\cite{me:classical},~\cite{me:number}.

Siegel's lemma is a fundamental principle in Diophantine approximations and transcendental number theory, which is a statement about the existence of points of small height in a vector space over a global field. This is an important instance of a general problem of finding rational points on varieties. We use height functions, which are essential in Diophantine geometry, as a measure of arithmetic complexity; we denote the homogeneous height on vectors by $H$, the inhomogeneous height by $h$, the height of a vector space by $\H$, and will define precisely our choice of heights below. Throughout this paper, we will write $K$ for either a number field, a function field (i.e. a finite algebraic extension of the field of rational functions in one variable over an arbitrary field), or the algebraic closure of one or the other. The following general version of Siegel's lemma was proved in~\cite{vaaler:siegel} if $K$ is a number field, in \cite{thunder} if $K$ is a function field, and in \cite{absolute:siegel} if $K$ is the algebraic closure of one or the other (see also \cite{absolute:siegel1} for an improved constant).

\begin{thm} [\cite{vaaler:siegel}, \cite{thunder}, \cite{absolute:siegel}, \cite{absolute:siegel1}] \label{siegel} Let $K$ be a number field, a function field, or the algebraic closure of one or the other. Let $V \subseteq K^N$ be an $L$-dimensional subspace, $1 \leq L \leq N$. Then there exists a basis $\bv_1,...,\bv_L$ for $V$ over $K$ such that
\begin{equation}
\label{siegel_bound}
\prod_{i=1}^L H(\bv_i) \leq C_K(L) \H(V),
\end{equation}
where $C_K(L)$ is a field constant defined by equation (\ref{CKL}) in section~\ref{notation} below. In fact, if $K$ is a number field or $\qbar$, then even more is true: there exists such a basis with
\begin{equation}
\label{siegel_bound1}
\prod_{i=1}^L H(\bv_i) \leq \prod_{i=1}^L h(\bv_i) \leq C_K(L) \H(V).
\end{equation}
\end{thm}

\noindent
The transition from projective height $H$ to inhomogeneous height $h$ in Theorem \ref{siegel} is quite straightforward over number fields (in other words, (\ref{siegel_bound1}) is a fairly direct corollary of (\ref{siegel_bound}) in the number field case and over $\qbar$). In section~\ref{inh_siegel} we prove the following function field analogue of (\ref{siegel_bound1}), which is one of the key tools we use to prove our main result, Theorem \ref{general_main}.

\begin{thm} \label{siegel_fi} Let $\kk_0$ be any perfect field and let $Y$ be a curve of genus $g$ over $\kk_0$. Let $K=\kk_0(Y)$ be the field of rational functions on $Y$ over $\kk_0$, and let $V \subseteq K^N$ be an $L$-dimensional subspace, $1 \leq L \leq N$. Then there exists a basis $\bu_1,...,\bu_L$ for $V$ over $K$ such that
\begin{equation}
\label{siegel_bound2}
\prod_{i=1}^L H(\bu_i) \leq \prod_{i=1}^L h(\bu_i) \leq \E_K(L) C_K(L) \H(V).
\end{equation}
where $\E_K(L)$ is as in (\ref{EKL}) and $C_K(L)$ is as in (\ref{CKL}).
\end{thm}

\begin{rem} \label{curve_rem} In Theorem \ref{siegel_fi}, as well as throughout this paper, all curves are always assumed to be smooth projective curves.
\end{rem}
\smallskip

An immediate consequence of Theorem \ref{siegel} is the existence of a nonzero point $\bv_1 \in V$ such that
\begin{equation}
\label{siegel_bound_1}
H(\bv_1) \leq \left( C_K(L) \H(V) \right)^{1/L}.
\end{equation}
The bounds of (\ref{siegel_bound}) and (\ref{siegel_bound_1}) are sharp in the sense that the exponents on $\H(V)$ are smallest possible. For many applications it is also important to have versions of Siegel's lemma with some additional algebraic conditions. One such example is the so called Faltings' version of Siegel's lemma, which guarantees the existence of a point of bounded norm in a vector space $V \subseteq \real^N$  outside of a subspace $U \subsetneq V$ (see \cite{faltings}, \cite{faltings:siegel}, and~\cite{faltings:siegel_1}). In \cite{me:classical} and \cite{me:number} I considered a more general related problem. Specifically, using the notation of Theorem \ref{siegel} in the case when $K$ is a number field, let $M \in \zed_{>0}$ and let $U_1,...,U_M$ be subspaces of $K^N$  such that $V \nsubseteq \bigcup_{i=1}^M U_i$. Then we can prove the existence of a non-zero point of small height in $V \setminus \bigcup_{i=1}^M U_i$ providing an explicit upper bound on the height of such a point. In particular, the main result of \cite{me:number} is the following.

\begin{thm} [\cite{me:number}] \label{numberfield:main} Let $K$ be a number field of degree $d$ with discriminant $\D_K$. Let $N \geq 2$ be an integer, $l=\left[ \frac{N}{2} \right]$, and let $V$ be a subspace of $K^N$ of dimension $L$, $1 \leq L \leq N$. Let $1 \leq s < L$ be an integer, and let $U_1,...,U_M$ be nonzero subspaces of $K^N$ with $\max_{1 \leq i \leq M} \{ \dim_K(U_i) \} \leq s$. There exists a point $\bx \in V \setminus \bigcup_{i=1}^M U_i$ such that
\begin{equation}
\label{numberfield:bound}
H(\bx) \leq B_K(N,L,s) \H(V)^d \left\{ \left( \sum_{i=1}^M \frac{1}{\H(U_i)^d} \right)^{\frac{1}{(L-s)d}} + M^{\frac{1}{(L-s)d+1}} \right\},
\end{equation}
where
\begin{equation}
\label{main:constant}
B_K(N,L,s) = 2^{L(d+3)} |\D_K|^{\frac{L}{2}} \left( (Ld)^L \binom{Nd}{ld}^{\frac{1}{2d}} \right)^{\frac{1}{L-s}}.
\end{equation}
\end{thm}

\noindent
If $\bx_1,\dots,\bx_L$ is any basis for $V$, then it is well known (see for instance Lemma 4.7 of \cite{absolute:siegel}) that
\begin{equation}
\label{basis}
\prod_{i=1}^L H(\bx_i) \geq N^{-\frac{L}{2}} \H(V).
\end{equation}
Let $M=1$, and take $U_1$ to be a subspace of $V$ of dimension $L-1$ generated by the vectors corresponding to the first $L-1$ successive minima of $V$ with respect to an adelic unit cube - these are precisely the vectors $\bv_1,\dots,\bv_{L-1}$ in Theorem \ref{siegel}. Then the smallest vector in $V \setminus U_1$ will be $\bv_L$ of Theorem \ref{siegel}. If we choose $V$ so that the first $L-1$ successive minima of $V$ are equal to 1, then (\ref{basis}) implies that $H(\bv_L) \geq N^{-L/2} \H(V)$. This shows that the dependence on $\H(V)$ in the upper bound of Theorem \ref{numberfield:main} is sharp in the case $K=\que$, however it is natural to expect the exponent on $\H(V)$ to be equal to 1 over any number field.

The proof of Theorem \ref{numberfield:main} relies on a counting argument. Write $O_K$ for the ring of integers of $K$, and view modules $V \cap O_K$ and $U_i \cap O_K$ for all $1 \leq i \leq M$ as lattices in $\real^{Nd}$ under the canonical embedding of $K$ into $\real^d$. Then one can count points of $V \cap O_K$ and $\bigcup_{i=1}^M U_i \cap O_K$ in a cube of side-length $2R$ centered at the origin in $\real^{Nd}$, and make $R$ sufficiently large so that there exists a point $\bx \in V \setminus \bigcup_{i=1}^M U_i$; now it is not difficult to estimate the height of this point. However, this argument does not extend to algebraically closed fields, since $\kbar$ does not embed into a finite-dimensional Euclidean space.

The main goal of this paper is to produce a generalization of Theorem \ref{numberfield:main} with optimal dependence on $\H(V)$ which holds just as well over $\qbar$ and over function fields. Let us say that $K$ is an {\it admissible field} if it is a number field, $\qbar$, or the field of rational functions on a smooth projective curve over a perfect field. We can now state our main result.

\begin{thm} \label{general_main} Let $K$ be an admissible field. Let $N \geq 2$ be an integer, and let $V$ be an $L$-dimensional subspace of $K^N$, $1 \leq L \leq N$. Let $J \geq 1$ be an integer. For each $1 \leq i \leq J$, let $k_i \geq 1$ be an integer and let
$$P_{i1}(X_1,\dots,X_N),\dots,P_{ik_i}(X_1,\dots,X_N)$$
be polynomials of respective degrees $m_{i1},\dots,m_{ik_i} \geq 1$, and define
\begin{equation}
\label{def_M}
M_i = \max_{1 \leq j \leq k_i} m_{ij}\ \forall\ 1 \leq i \leq J,\ \ M = \sum_{i=1}^J M_i.
\end{equation}
Let
$$Z_K(P_{i1},\dots,P_{ik_i}) = \{ \bx \in K^N : P_{i1}(\bx) = \dots = P_{ik_i}(\bx) = 0 \},$$
and define $\Z_K = \bigcup_{i=1}^J Z_K(P_{i1},\dots,P_{ik_i})$. Suppose that $V \nsubseteq \Z_K$. Let
\begin{equation}
\label{delta}
\delta = \left\{ \begin{array}{ll}
1 & \mbox{if $K$ is a number field or $\qbar$} \\
0 & \mbox{otherwise}.
\end{array}
\right.
\end{equation}
Then there exists a point $\bx \in V \setminus \Z_K$ such that
\begin{equation}
\label{gen_bnd_1}
H(\bx) \leq h(\bx) \leq  L^{\delta} \E_K(L)^{1-\delta} A_K(L,M) C_K(L) \H(V),
\end{equation}
where $C_K(L)$ is as in (\ref{CKL}), $A_K(L,M)$ is as in (\ref{AKLM}), and $\E_K(L)$ is as in (\ref{EKL}).
\end{thm}

\noindent
In case $K$ is a function field over a finite field, the constant $A_K(L,M) C_K(L)$ in the upper bound of (\ref{gen_bnd_1}) can be slightly simplified: see Remark \ref{fnct_01} in section~\ref{notation} below. It should also be remarked that all the ingredients of our method (Lemma \ref{sum_height}, Theorem \ref{combin_null_1}, and Lemma \ref{count_f}) except for one (Theorem \ref{siegel_fi}) also work over the algebraic closure of a function field. Hence we state and prove our results in their most general form whenever possible. 

An immediate corollary of Theorem \ref{general_main} is the following extension of Theorem~\ref{numberfield:main}.

\begin{cor} \label{main_cor} Let $K$ be an admissible field. Let $N \geq 2$ be an integer, and let $V$ be an $L$-dimensional subspace of $K^N$, $1 \leq L \leq N$. Suppose that $M \geq 1$ is an integer and let $U_1,...,U_M$ be subspaces of $K^N$  such that $V \nsubseteq \bigcup_{i=1}^M U_i$. Then there exists a point $\bx \in V \setminus \bigcup_{i=1}^M U_i$ satisfying (\ref{gen_bnd_1}) above. In particular, in case $K$ is a number field,
\begin{equation}
\label{gen_bnd_2}
H(\bx) \leq h(\bx) \leq \sqrt{2} L |\D_K|^{\frac{L+1}{2d}} M^{\frac{1}{d}} \H(V).
\end{equation}
\end{cor}

\proof
Since $V \nsubseteq \bigcup_{i=1}^M U_i$, there exist subspaces $\ubar_1,\dots,\ubar_M$ of $K^N$ of dimension $N-1$ such that $U_i \subseteq \ubar_i$ for each $1 \leq i \leq M$, and $V \nsubseteq \bigcup_{i=1}^M \ubar_i$. Let
$$\L_1(X_1,\dots,X_N),\dots,\L_M(X_1,\dots,X_N) \in K[X_1,\dots,X_N]$$
be linear forms such that $\ubar_i = \{ \bx \in K^N : \L(\bx) = 0 \}$ for each $1 \leq i \leq M$, and define
$$P(X_1,\dots,X_N) = \prod_{i=1}^M \L_i(X_1,\dots,X_N) \in K[X_1,\dots,X_N].$$
Then $P$ is a polynomial of degree $M$, and $Z_K(P) = \bigcup_{i=1}^M \ubar_i$. Now the statement of the corollary follows from Theorem \ref{general_main}.
\endproof

\noindent
Notice that although the bound of Corollary \ref{main_cor} does not uniformly overrule Theorem \ref{numberfield:main} (in particular, there is no dependence on the heights of $U_i$ and the dependence on $M$ is not as good as in Theorem \ref{numberfield:main}), it exhibits the optimal exponent on $\H(V)$, better dependence on $N,L, d, \D_K$, is easier to use (compare (\ref{numberfield:bound}) with (\ref{gen_bnd_2})), and extends to $\qbar$ and over function fields, which is a serious advantage.

Our argument builds on the method of \cite{me:classical} and \cite{me:number}. We use a variation of the Combinatorial Nullstellensatz of N. Alon \cite{alon} along with a counting mechanism. Loosely speaking, the Combinatorial Nullstellensatz is the general principle that a polynomial of degree $M$ in $N$ variables cannot uniformly vanish on certain sets of points in $K^N$, which are built as rectangular grids of cardinality $\gg M^N$. A similar principle has been used in \cite{me:classical} and \cite{me:number}. The main novelty in our approach is that we restrict this principle to points in a fixed vector space, and then reduce the main counting argument in the number field case to points of $O_K$ viewed as a full-rank lattice in $\real^d$. In the function field case, we use a construction of FML lattices as in \cite{tsfasman}, pp. 578--583, combined with a lemma from \cite{me:classical} to produce a counting mechanism; we also discuss a possible alternative construction in Remark \ref{Thunder_JNT}. This, along with an application of Siegel's lemma with inhomogeneous heights (Theorems \ref{siegel} and \ref{siegel_fi}), allows us to produce a sharper estimate. The fact that Combinatorial Nullstellensatz applies over any field (or any sufficiently large subset of a field, for that matter) allows us to extend our results over $\kbar$. The dependence on $M$ in the number field case of Theorem \ref{general_main} is optimal in the sense that if $M^{1/d}$ is replaced by a smaller power of $M$ then the corresponding rectangular grid in Combinatorial Nullstellensatz is not sufficiently large, so that the polynomial in question may vanish identically on it (see Remark \ref{M_optimal} below for an actual example).
\smallskip

As a side product of the counting part of our method, we are also able to produce a uniform lower bound on the number of algebraic integers of bounded height in a number field $K$. The subject of counting {\it algebraic numbers} of bounded height has been started by the famous asymptotic formula of Schanuel \cite{schanuel}. Some explicit upper and lower bounds have also been produced later, for instance by Schmidt \cite{schm1}, \cite{schm2}. Recently a new sharp upper bound has been given by Loher and Masser \cite{loher:masser}. Here we can produce the following lower bound estimate for the number of {\it algebraic integers}.

\begin{cor} \label{lower_bnd} Let $K$ be a number field of degree $d$ over $\que$ with discriminant $\D_K$ and $r_1$ real embeddings. Let $O_K$ be its ring of integers. For all $R \geq \left( 2^{r_1} |D_K| \right)^{1/2}$,
\begin{equation}
\label{count_lower}
\left( 2^{r_1} |\D_K| \right)^{-1/2} R^d < \left| \{ x \in O_K\ :\ h(x) \leq R \} \right|.
\end{equation}
\end{cor}

\begin{rem} \label{loher_masser} For comparison, Loher and Masser \cite{loher:masser} produce the following upper bound on the number of {\it algebraic numbers} of bounded height in $K$ (using notation of Corollary \ref{lower_bnd} above):
$$\left| \{ x \in K\ :\ h(x) \leq R \} \right| \leq 31(d \log d) R^{2d}.$$
\end{rem}
\smallskip

The paper is structured as follows: in section~\ref{notation} we set notation, define heights, and recall Lemma \ref{sum_height}, which is a useful property of heights for our purposes; in section~\ref{inh_siegel} we prove a function field version of Siegel's lemma with inhomogeneous heights; in section~\ref{null} we prove Theorem \ref{combin_null_1}, a version of Combinatorial Nullstellensatz on a vector space required for our argument; in section~\ref{number_count} we prove Lemma \ref{count}, which is our main counting lemma in the number field case, and derive Corollary \ref{lower_bnd} from it; in section~\ref{funct_count} we prove Lemma \ref{count_f}, the counting lemma over a function field; in section~\ref{proof} we prove Theorem \ref{general_main}; in section~\ref{twisted} we discuss how our results can be extended to inequalities involving twisted height.
\smallskip

\begin{rem} \label{gaudron} The original version of this paper was posted on the arXiv in August 2008 (arXiv:0808.2476) and appeared in the Max-Planck-Institut f\"{u}r Mathematik preprint series (Bonn, Germany). The author was recently informed that E. Gaudron later published a paper \cite{gaudron}, in which he uses a different method to improve the bounds of the author's result in \cite{me:number}, quoted as Theorem \ref{numberfield:main} above. As in Theorem \ref{numberfield:main}, Gaudron's bound depends on heights of subspaces and only applies to the situation of a collection of linear subspaces of a vector space over a number field; it can be viewed as a different version of our Corollary \ref{main_cor}. It should be remarked that Gaudron's paper does not treat the general situation of Theorem \ref{general_main}, which is the main result of the present work.
\end{rem}
\bigskip

\section{Notation and heights}
\label{notation}

We start with some notation.  Throughout this paper, $K$ will either be a number field (finite extension of $\que$), a function field, or algebraic closure of one or the other; in fact, for the rest of this section, unless explicitly specified otherwise, we will assume that $K$ is either a number field or a function field, and will write $\kbar$ for its algebraic closure. By a function field we will always mean a finite algebraic extension of the field $\kk = \kk_0(t)$ of rational functions in one variable over a field $\kk_0$, where $\kk_0$ can be any field. When $K$ is a number field, clearly $K \subset \kbar = \qbar$; when $K$ is a function field, $K \subset \kbar = \kkbar$, the algebraic closure of $\kk$. In the number field case, we write $d = [K:\que]$ for the global degree of $K$ over $\que$; in the function field case, the global degree is $d = [K:\kk]$, and we also define the effective degree of $K$ over $\kk$ to be
$$\mm(K,\kk) = \frac{[K:\kk]}{[K_0:\kk_0]},$$
where $K_0$ is the algebraic closure of $\kk_0$ in $K$. If $K$ is a number field, we let $\D_K$ be its discriminant, $\omega_K$ the number of roots of unity in $K$, $r_1$ its number of real embeddings, and $r_2$ its number of conjugate pairs of complex embeddings, so $d=r_1+2r_2$. If $K$ is a function field, we will also write $g = g(K)$ for the genus of $K$, as defined by the Riemann-Roch theorem (see \cite{thunder} for details). We will also need to define $\dd = \dd(K) := \min_{v \in M(K)} \deg(v)$, and let
\begin{equation}
\label{EKL}
\E_K(L) = e^{\frac{\dd g L}{d}}.
\end{equation}
We will distinguish two cases: if $K$ is a function field, we say that it is of {\it finite type $q$} if its subfield of constants is a finite field $\ff_q$ for some prime power $q$, and we say that it is of {\it infinite type} if its subfield of constants is infinite. If $K$ is a function field of finite type $q$, then there exists a unique smooth projective curve $Y$ over $\ff_q$ such that $K = \ff_q(Y)$ is the field of rational functions on $Y$. In this case, we will write $n(K) = |Y(\ff_q)|$ for the number of points of $Y$ over $\ff_q$, and $h_K$ for the number of divisor classes of degree zero (which is precisely the cardinality of the Jacobian of $Y$ over $\ff_q$). We can now define the field constant $C_K(L)$, which appears in Theorems \ref{siegel} and \ref{general_main}:
\begin{equation}
\label{CKL}
C_K(L) = \left\{ \begin{array}{ll}
\left( \left( \frac{2}{\pi} \right)^{r_2} |\D_K| \right)^{\frac{L}{2d}} & \mbox{if $K$ is a number field} \\
\exp\left(\frac{(g(K)-1+\mm(K,\kk))L}{\mm(K,\kk)}\right) & \mbox{if $K$ is a function field} \\
e^{\frac{L(L-1)}{4}} + \eps & \mbox{if $K = \qbar$; here we can take any $\eps >0$} \\
1+\eps & \mbox{if $K = \overline{\kk}$; here we can take any $\eps >0$},
\end{array}
\right.
\end{equation}
and the constant $A_K(L,M)$, which appears in the statement of Theorem \ref{general_main}:
\begin{equation}
\label{AKLM}
A_K(L,M) = \left\{ \begin{array}{ll}
\left( M \sqrt{2^{r_1} |\D_K|} \right)^{\frac{1}{d}} & \mbox{if $K$ is a number field with $\omega_K \leq M$} \\
e^{R_K(M)}  & \mbox{if $K$ is a function field of finite type $q \leq M$} \\
1 & \mbox{otherwise,}
\end{array}
\right.
\end{equation}
for all integers $L,M \geq 1$, where for a function field $K$ of finite type $q \leq M$ we define
\begin{equation}
\label{RK}
R_K(M) = \frac{n(K)-1}{2} \left( (M-q+2) h_K \sqrt{n(K)} \right)^{\frac{1}{n(K)-1}} + h_K (n(K)-1) \sqrt{n(K)}.
\end{equation}

\begin{rem} \label{fnct_01} Let $Y$ be a smooth projective curve of genus $g$ over $\ff_q$. Then Hasse-Weil-Serre bound (see for instance Theorem 2.3.16 on p. 178 of \cite{tsfasman}) gives
\begin{equation}
\label{hasse_weil}
n(K) \leq q+1+g \left[ 2\sqrt{q} \right],
\end{equation}
where $[\ ]$ stands for the integer part function. In case $g=0$ we also have $h_K=1$, and if $g=1$ we have $h_K \leq n(K) \leq q+1+\left[ 2\sqrt{q} \right]$ (see (\ref{dltn}) below, which gives a bound on $h_K$ in terms of $n(K)$ and the genus). These observations may help to simplify the formula (\ref{RK}) for $R_K(M)$.
\end{rem}
\smallskip

Next we discuss absolute values on $K$. Let $M(K)$ be the set of places of $K$. For each place $v \in M(K)$ we write $K_v$ for the completion of $K$ at $v$ and let $d_v$ be the local degree of $K$ at $v$, which is $[K_v:\que_v]$ in the number field case, and $[K_v:\kk_v]$ in the function field case. In any case, for each place $u$ of the ground field, be it $\que$ or $\kk$, we have
\begin{equation}
\sum_{v \in M(K), v|u} d_v = d.
\end{equation}

If $K$ is a number field, then for each place $v \in M(K)$ we define the absolute value $|\ |_v$ to be the unique absolute value on $K_v$ that extends either the usual absolute value on $\real$ or $\cee$ if $v | \infty$, or the usual $p$-adic absolute value on $\que_p$ if $v|p$, where $p$ is a prime. For each finite place $v \in M(K)$, $v \nmid \infty$, we define the {\it local ring of $v$-adic integers} $\OO_v = \{ x \in K : |x|_v \leq 1 \}$, whose unique maximal ideal is $\MM_v =  \{ x \in K : |x|_v < 1 \}$. Then $O_K = \bigcap_{v \nmid \infty} \OO_v$. 

If $K$ is a function field, then all absolute values on $K$ are non-archimedean. For each $v \in M(K)$, let $\OO_v$ be the valuation ring of $v$ in $K_v$ and $\MM_v$ the unique maximal ideal in $\OO_v$. We choose the unique corresponding absolute value $|\ |_v$ such that:
\begin{trivlist}
\item (i) if $1/t \in \MM_v$, then $|t|_v = e$,
\item (ii) if an irreducible polynomial $p(t) \in \MM_v$, then $|p(t)|_v = e^{-\deg(p)}$.
\end{trivlist}

\noindent
In both cases, for each non-zero $a \in K$ the {\it product formula} reads
\begin{equation}
\label{product_formula}
\prod_{v \in M(K)} |a|^{d_v}_v = 1.
\end{equation} 

We extend absolute values to vectors by defining the local heights. For each $v \in M(K)$ define a local height $H_v$ on $K_v^N$ by
$$H_v(\bx) = \max_{1 \leq i \leq N} |x_i|^{d_v}_v,$$
for each $\bx \in K_v^N$. Also, for each $v | \infty$ we define another local height
$$\H_v(\bx) = \left( \sum_{i=1}^N |x_i|_v^2 \right)^{d_v/2}.$$
Then we can define two slightly different global height functions on $K^N$:
\begin{equation}
\label{global_heights}
H(\bx) = \left( \prod_{v \in M(K)} H_v(\bx) \right)^{1/d},\ \ \H(\bx) = \left( \prod_{v \nmid \infty} H_v(\bx) \times \prod_{v | \infty} \H_v(\bx) \right)^{1/d},
\end{equation}
for each $\bx \in K^N$. These height functions are {\it homogeneous}, in the sense that they are defined on projective space thanks to the product formula (\ref{product_formula}): $H(a \bx) = H(\bx)$ and $\H(a \bx) = \H(\bx)$ for any $\bx \in K^N$ and $0 \neq a \in K$. It is easy to see that
$$H(\bx) \leq \H(\bx) \leq \sqrt{N} H(\bx).$$
Notice that in case $K$ is a function field, $M(K)$ contains no archimedean places, and so $H(\bx) = \H(\bx)$ for all $\bx \in K^N$. We also define the {\it inhomogeneous} height
$$h(\bx) = H(1,\bx),$$
which generalizes Weil height on algebraic numbers: for each $\alpha \in K$, define
$$h(\alpha) = \prod_{v \in M(K)} \max \{ 1, |\alpha|_v \}^{d_v/d}.$$
Clearly, $h(\bx) \geq H(\bx)$ for each $\bx \in K^N$. All our inequalities will use heights $H$ and $h$ for vectors, however we use $\H$ to define the conventional Schmidt height on subspaces in the manner described below. This choice of heights coincides with \cite{vaaler:siegel} and \cite{me:number}.
\smallskip

We extend both heights $H$ and $\H$ to polynomials by viewing them as height functions of the coefficient vector of a given polynomial. We also define a height function on subspaces of $K^N$. Let $V \subseteq K^N$ be a subspace of dimension $L$, $1 \leq L \leq N$. Choose a basis $\bx_1,...,\bx_L$ for $V$, and write $X = (\bx_1\ ...\ \bx_L)$ for the corresponding $N \times L$ basis matrix. Then 
$$V = \{ X \bt : \bt \in K^L \}.$$
On the other hand, there exists an $(N-L) \times N$ matrix $A$ with entries in $K$ such that 
$$V = \{ \bx \in K^N : A \bx = 0 \}.$$
Let $\I$ be the collection of all subsets $I$ of $\{1,...,N\}$ of cardinality $L$. For each $I \in \I$ let $I'$ be its complement, i.e. $I' = \{1,...,N\} \setminus I$, and let $\I' = \{ I' : I \in \I\}$. Then 
$$|\I| = \binom{N}{L} = \binom{N}{N-L} = |\I'|.$$
For each $I \in \I$, write $X_I$ for the $L \times L$ submatrix of $X$ consisting of all those rows of $X$ which are indexed by $I$, and $_{I'} A$ for the $(N-L) \times (N-L)$ submatrix of $A$ consisting of all those columns of $A$ which are indexed by $I'$. By the duality principle of Brill-Gordan \cite{gordan:1} (also see Theorem 1 on p. 294 of \cite{hodge:pedoe}), there exists a non-zero constant $\gamma \in K$ such that
\begin{equation}
\label{duality}
\det (X_I) = (-1)^{\varepsilon(I')} \gamma \det (_{I'} A),
\end{equation}
where $\varepsilon(I') = \sum_{i \in I'} i$. Define the vectors of {\it Grassmann coordinates} of $X$ and $A$ respectively to be 
$$Gr(X) = (\det (X_I))_{I \in \I} \in K^{|I|},\ \ Gr(A) = (\det (_{I'} A))_{I' \in \I'} \in K^{|I'|},$$
and so by (\ref{duality}) and (\ref{product_formula})
$$\H(Gr(X)) = \H(Gr(A)).$$
Define the height of $V$ denoted by $\H(V)$ to be this common value. This definition is legitimate, since it does not depend on the choice of the basis for $V$. In particular, notice that if 
$$\L(X_1,...,X_N) = \sum_{i=1}^N q_i X_i \in K[X_1,...,X_N]$$
is a linear form with a non-zero coefficient vector $\bq \in K^N$, and $V = \{ \bx \in K^N : \L(\bx) = 0 \}$ is an $(N-1)$-dimensional subspace of $K^N$, then
\begin{equation}
\label{1.4}
\H(V) = \H(\L) = \H(\bq).
\end{equation}
An important observation is that due to the normalizing exponent $1/d$ in (\ref{global_heights}) all our heights are {\it absolute}, meaning that they do not depend on the number field or function field of definition, hence are well defined over~$\kbar$.
\smallskip

We will also need the following basic property of heights.

\begin{lem} \label{sum_height} For $\xi_1,...,\xi_L \in \kbar$ and $\bx_1,...,\bx_L \in \kbar^N$,
$$H \left( \sum_{i=1}^L \xi_i \bx_i \right) \leq h \left( \sum_{i=1}^L \xi_i \bx_i \right) \leq L^{\delta} h(\bxi) \prod_{i=1}^L h(\bx_i),$$
where $\bxi = (\xi_1,...,\xi_L) \in \kbar^L$, and $\delta$ is as in (\ref{delta}) above.
\end{lem}

We are now ready to proceed.
\bigskip

\section{Siegel's lemma over a function field}
\label{inh_siegel}

In this section we produce a version of Siegel's lemma with inhomogeneous heights over fields of rational functions of smooth projective curves, which is a function field analogue of (\ref{siegel_bound1}). Let all notation be as in section~\ref{notation} above. We now prove Theorem \ref{siegel_fi}.

\begin{proof}
[Proof of Theorem \ref{siegel_fi}]
Let $Y$ be a smooth projective curve of genus $g$ over a perfect field $\kk_0$ and $K=\kk_0(Y)$, then $g(K)=g$. Let $\bx_1,\dots,\bx_L$ be a basis for $V$ over $K$ satisfying (\ref{siegel_bound}) of Theorem \ref{siegel}. 

Fix $1 \leq i \leq L$, and for each $1 \leq j \leq N$ and $v \in M(K)$, let $\ord_v (x_{ij})$ be the order of $x_{ij}$ at the place $v$; clearly, for each $1 \leq i \leq L$, $1 \leq j \leq N$,  $\ord_v (x_{ij}) \neq 0$ at only finitely many places $v \in M(K)$. Then let $v_1,\dots,v_s$ be the places of $K$ at which $\ord_v (x_{ij}) \neq 0$ for some $1 \leq j \leq N$. As in section~\ref{notation}, for each $1 \leq m \leq s$
$$\OO_{v_m} = \{ x \in K : \ord_{v_m} (x) \leq 0 \}$$
is the valuation ring at $v_m$ with the unique maximal ideal 
$$\MM_{v_m} = \{ x \in K : \ord_{v_m} (x) < 0 \},$$
and let us write $\kk_0(v_m)$ for the residue field $\OO_{v_m}/\MM_{v_m}$. Clearly $\kk_0^* \subseteq \OO_{v_m} \setminus \MM_{v_m}$, so $\kk_0(v_m)$ is a field extension of $\kk_0$. By Exercise 2.3.1 on p. 171 of \cite{tsfasman}, $\delta_m := [\kk_0(v_m):\kk_0]$ is finite. Following the construction on p.171 of \cite{tsfasman}, we say that each $v_m$ determines a point $P(v_m)$ of $Y$ of degree $\delta_m$ (we will also denote this degree by $\deg_{\kk_0}(v_m)$), and write $\ybar$ for the closure of the curve $Y$ over $\overline{\kk_0}$. Then the Galois orbit of $P(v_m)$ over $\kk_0(v_m)$ consists of $\delta_m$ points $P_1(v_m),\dots,P_{\delta_m}(v_m)$ on $\ybar$, i.e.
$$\left\{ \sigma(P(v_m)) : \sigma \in \Gal(\overline{\kk_0}/\kk_0) \right\} = \{ P_1(v_m),\dots,P_{\delta_m}(v_m) \}.$$
We will say that the points $P_1(v_m),\dots,P_{\delta_m}(v_m)$ lie over $v_m$. Since $\kk_0$ is perfect, $\overline{\kk_0}$ is separable over $\kk_0$, and so $P_{k}(v_{m_1}) = P_{l}(v_{m_2})$ if and only if $m_1=m_2$ and $k=l$. Define the divisor of $\bx_i$ over $\overline{\kk_0}$ by the formal sum
$$\div (\bx_i) = \sum_{m=1}^s \left( - \min_{1 \leq j \leq N} \ord_{v_m} (x_{ij}) \right) (P_1(v_m) + \dots + P_{\delta_m}(v_m)),$$
then as usual
$$\deg (\div (\bx_i)) = - \sum_{m=1}^s \delta_m \min_{1 \leq j \leq N} \ord_{v_m} (x_{ij}).$$
In the same manner, each element $f \in \kk_0(Y)$ defines a principal divisor
$$(f) = \sum_{v \in M(\kk_0(Y))} \left( \ord_v (f) \right) (P_1(v) + \dots + P_{\deg_{\kk_0}(v)}(v)),$$
so that $\deg (f) =  \sum_{v \in M(\kk_0(Y))} \deg_{\kk_0}(v) \ord_v (f) = 0$. In particular notice that
\begin{equation}
\label{div_0}
0 = - \sum_{m=1}^s \delta_m \ord_{v_m} (x_{i1}) \leq - \sum_{m=1}^s \delta_m \min_{1 \leq j \leq N} \ord_{v_m} (x_{ij}) = \deg (\div (\bx_i)).
\end{equation}
Let $w \in M(K)$ be a place with minimal degree, then its degree is a constant depending on $K$ only, so define $\dd = \deg(w)$. Then
$$\deg(\div(\bx_i)+g w) \geq g,$$
and an immediate implication of the Riemann-Roch theorem (see for instance Theorem 2.2.17 on p. 150 of \cite{tsfasman}) is that there exists $f_i \in \overline{\kk_0}(Y)$ such that the divisor $\div(\bx_i) + g w+ (f_i)$ is effective. Then, by Exercise 2.3.6 on p. 174 of \cite{tsfasman}, there in fact exists such $f_i \in K$, so
$$\deg(v) \left( - \min_{1 \leq j \leq N} \ord_v (x_{ij}) + \ord_v (f_i) \right)  \geq 0$$
for all $v \in M(K) \setminus \{ w \}$, and
$$\dd \left( - \min_{1 \leq j \leq N} \ord_w (x_{ij}) + g + \ord_w (f_i) \right)  \geq 0,$$
where $\deg(v), \dd \geq 1$. Now notice that for each $v \in M(K)$,
\begin{eqnarray*}
H_v(\bx_i) & = & \max_{1 \leq j \leq N} |x_{ij}|^{d_v}_v = \max_{1 \leq j \leq N} e^{-\ord_v (x_{ij}) \deg(v)} \\
& = & \exp \left( - \deg(v) \min_{1 \leq j \leq N} \ord_v (x_{ij}) \right).
\end{eqnarray*}
Then define $\bu_i = \frac{1}{f_i} \bx_i$, and notice that
$$H_v(\bu_i) = \exp \left( - \deg(v) \min_{1 \leq j \leq N} \ord_v (x_{ij}) + \deg(v) \ord_v (f_i) \right) \geq 1,$$
for all $v \in M(K) \setminus \{ w \}$, and
$$e^{\dd g} H_w(\bu_i) = \exp \left( \dd \left( - \min_{1 \leq j \leq N} \ord_w (x_{ij}) + g + \ord_w (f_i) \right) \right) \geq 1.$$
Therefore
\begin{equation}
\label{inh1}
h(\bu_i)  \leq e^{\frac{\dd g}{d}} H(\bu_i) = e^{\frac{\dd g}{d}} \left( \prod_{v \in M(K)} \left| \frac{1}{f_i} \right|^{d_v}_v H_v(\bx_i) \right)^{1/d} = e^{\frac{\dd g}{d}} H(\bx_i),
\end{equation}
by the product formula. Then combining (\ref{inh1}) with (\ref{siegel_bound}), we see that there exists a basis $\bu_1,...,\bu_L$ for $V$ over $K$ such that
$$\prod_{i=1}^L H(\bu_i) \leq \prod_{i=1}^L h(\bu_i) \leq e^{\frac{\dd gL}{d}} \prod_{i=1}^L H(\bx_i) \leq e^{\frac{\dd gL}{d}} C_K(L) \H(V).$$
This completes the proof.
\end{proof}

\begin{rem} \label{billy} In the proof above, the argument introducing the convenient field constant $\dd$ which allows one to deal with divisors of small degree in case of fields of genus larger than one was suggested to me by Wai Kiu Chan.
\end{rem}
\bigskip

\section{Combinatorial Nullstellensatz}
\label{null}

In \cite{alon} the following lemma is proved (compare with Lemma 2.1 of \cite{me:classical}, which is an immediate corollary of Lemma 1 on p. 261 of \cite{cass:geom}).

\begin{lem} [\cite{alon}] \label{combin_null} Let $P(X_1,\dots,X_N)$ be a polynomial in $N$ variables with coefficients in an arbitrary field $\ff$. Suppose that $\deg_{X_i} P \leq t_i$ for $1 \leq i \leq N$, and let $S_i \subset \ff$ be a set of at least $t_i+1$ distinct elements of $\ff$. If $P(\bxi) = 0$ for all $N$-tuples
$$\bxi = (\xi_1,\dots,\xi_N) \in S_1 \times \dots \times S_N,$$
then $P \equiv 0$.
\end{lem}

\noindent
We will refer to this lemma as Combinatorial Nullstellensatz (Alon uses this name for a slightly different related result, which is derived from this lemma). We use this lemma to derive a somewhat more specialized version of such a result with restriction to a vector space.

\begin{thm} \label{combin_null_1} Let $P(X_1,\dots,X_N)$ be a polynomial in $N$ variables with coefficients in an arbitrary field $\ff$. Suppose that $\deg P \leq M$, and let $S_i \subset \ff$ be a set of at least $M+1$ distinct elements of $\ff$ for each $1 \leq i \leq N$. Let $\bv_1,\dots,\bv_L$ be vectors in $\ff^N$, $1 \leq L \leq N$, and let $V = \spn_{\ff} \{ \bv_1,\dots,\bv_L \}$ be a subspace of $\ff^N$. Write $S = S_1 \times \dots \times S_L$, and for each $L$-tuple $\bxi = (\xi_1,\dots,\xi_L) \in S$, let $\bv(\bxi) = \sum_{i=1}^L \xi_i \bv_i$. If $P(\bv(\bxi)) = 0$ for all $\bxi \in S$, then $P$ is identically 0 on $V$. 
\end{thm}

\proof
Assume that $P$ is not identically zero on $V$, so there exists $\bx \in V$ such that $P(\bx) \neq 0$. We will show that there must exist $\bxi \in S$ such that $P(\bv(\bxi)) \neq 0$. Let
$$A = ( \bv_1 \dots \bv_L\ \bo \dots \bo)$$
be the $N \times N$ matrix the first $L$ columns of which are the vectors $\bv_1,\dots,\bv_L$, and the remaining $N-L$ columns are zero vectors. Write $\bX = (X_1,\dots,X_N)$ for the variable vector, and define the restriction of $P$ to $V$ with respect to the spanning set $\{ \bv_1,\dots,\bv_L\}$ by
$$P_V(X_1,\dots,X_L) = P(A\bX^t).$$
Notice that if $\bv(\bxi) = \sum_{i=1}^L \xi_i \bv_i$ for some $\bxi = (\xi_1,\dots,\xi_L) \in \ff^L$, then $P(\bv(\bxi)) = P_V(\bxi)$. Since $P$ is not identically zero on $V$, there must exist $\bxi \in \ff^L$ such that $P_V(\bxi) \neq 0$. Moreover, for each $1 \leq i \leq L$, $\deg_{X_i} P_V \leq \deg P \leq M$. Therefore by Lemma \ref{combin_null}, there exists $\bxi \in S$ such that 
$$P_V(\bxi) = P(\bv(\bxi)) \neq 0.$$
This completes the proof.
\endproof
\bigskip

\section{A counting mechanism: number field case}
\label{number_count}

Here we produce a certain refinement of Theorem 0 on p. 102 of \cite{lang} with explicit constants (also compare with  Lemma 4.1 of \cite{me:number}), which will serve as our main counting mechanism in the number field case. We start by recalling Lemma 2.1 of \cite{me:number}.

\begin{lem} [\cite{me:number}] \label{2.4.1} For a real number $R \geq 1$, let
\begin{equation}
\label{cube}
C^n_R = \{\bx \in \real^n : \max_{1 \leq i \leq n} |x_i| \leq R \}
\end{equation}
be a cube in $\real^n$, $n \geq 1$, centered at the origin with sidelength $2R$. Let $\Lambda$ be a lattice of full rank in $\real^n$ of determinant $\Delta$ such that there exists a positive constant $c$ and an uppertriangular basis matrix $A = (a_{ij})_{1 \leq i,j \leq n}$ of $\Lambda$ with diagonal entries $a_{ii} \geq c$ for all $1 \leq i \leq n$. Assume that $2R \geq \max \left\{ \frac{\Delta}{c^{n-1}}, c \right\}$. Then for each point $\bz$ in $\real^n$ we have
\begin{eqnarray}
\label{cube:lattice}
\left( \frac{2R c^{n-1}}{\Delta} - 1 \right)\left( \frac{2R}{c} - 1 \right)^{n-1} & \leq & |\Lambda \cap (C_R^n+\bz)| \nonumber \\
& \leq & \left( \frac{2R c^{n-1}}{\Delta} + 1 \right)\left( \frac{2R}{c} + 1 \right)^{n-1}.
\end{eqnarray}
\end{lem}

\noindent
For our number field $K$, define the set
\begin{equation}
\label{SRK}
S_R(K) = \left\{ x \in O_K\ :\ |x|_v \leq R\ \forall\ v | \infty \right\},
\end{equation}
where $R \geq 1$ is a real number (compare with the set $S_M(K)$ in the proof of Lemma 4.1 in \cite{me:number}). We use Lemma \ref{2.4.1} to prove the following estimate, which will be essential in the proof of Theorem \ref{general_main}.

\begin{lem} \label{count} For all $R \geq \left( 2^{r_1} |D_K| \right)^{1/2}$,
\begin{equation}
\label{count_mech}
\left( 2^{r_1} |\D_K| \right)^{-1/2} R^d < |S_R(K)| < 2^{2d+1/2} \left( 2^{r_1} |\D_K| \right)^{-1/2} R^d.
\end{equation}
\end{lem}

\proof
As in \cite{me:number}, let 
$$\sigma_1,...,\sigma_{r_1},\tau_1,...,\tau_{r_2},\tau_{r_2+1},...,\tau_{2r_2}$$
be the embeddings of $K$ into $\cee$ with $\sigma_1,...,\sigma_{r_1}$ being real embeddings and $\tau_j,\tau_{r_2+j} = \bar{\tau}_j$ for each $1 \leq j \leq r_2$ being the pairs of complex conjugate embeddings. For each $x \in K$ and each complex embedding $\tau_j$, write $\tau_{j1}(x) = \Re(\tau_j(x))$ and $\tau_{j2}(x) = \Im(\tau_j(x))$, where $\Re$ and $\Im$ stand respectively for real and imaginary parts of a complex number. We will view $\tau_j(x)$ as a pair $(\tau_{j1}(x), \tau_{j2}(x)) \in \real^2$. Then $d=r_1+2r_2$, and we define an embedding
$$\sigma=(\sigma_1,...,\sigma_{r_1}, \tau_1,...,\tau_{r_2}): K \longrightarrow K_{\infty},$$
where
$$K_{\infty} = \prod_{v|\infty} K_v = \prod_{v|\infty} \real^{d_v} = \real^d,$$
since $\sum_{v|\infty} d_v = d$. Then $\Lambda := \sigma(O_K)$ is a lattice of full rank in $\real^{d}$. Let us write $M_{\infty}(K)$ for the set of archimedean places of $K$, then
$$M_{\infty}(K) = \{ v_1, \dots, v_{r_1}, w_1,\dots,w_{r_2} \},$$
where for each $x \in K$, $1 \leq i \leq r_1$, $1 \leq j \leq r_2$,
$$|x|_{v_i} = |\sigma_i(x)|_{\infty},\ |x|_{w_j} = |\tau_j(x)|_{\infty},$$
where $|\ |_{\infty}$ stands for the usual absolute value on $\cee$. Therefore for each $x \in O_K$,
$$\sigma(x) = (\sigma_1(x),\dots,\sigma_{r_1}(x),\tau_{11}(x),\tau_{12}(x),\dots,\tau_{r_21}(x),\tau_{r_22}(x)) \in \Lambda,$$
and if $x \in S_R(K)$, then for each $1 \leq i \leq r_1$, $|\sigma_i(x)|_{\infty} \leq R$, and for each $1 \leq j \leq r_2$, $\sqrt{ \tau_{j1}(x)^2 + \tau_{j2}(x)^2 } \leq R$, thus
\begin{equation}
\label{c1}
\Lambda \cap C_{R/\sqrt{2}}^d \subseteq \sigma(S_R(K)) \subseteq \Lambda \cap C_R^d,
\end{equation}
and since $\sigma$ is injective,
\begin{equation}
\label{c2}
|\Lambda \cap C_{R/\sqrt{2}}^d| \leq |S_R(K)| \leq |\Lambda \cap C_R^d|.
\end{equation}

Now if $x \in O_K$, then $|x|_v \leq 1$ for all $v \nmid \infty$, and so $|x|_v \geq 1$ for at least one $v | \infty$, call this place $v_*$. If $v_*$ is real, say $v_*=v_i$ for some $1 \leq i \leq r_1$, then $|\sigma_j(x)|_{\infty} \geq 1$. If $v_*$ is complex, say $v_*=w_j$ for some $1 \leq j \leq r_2$, then $\sqrt{\tau_{j1}(x)^2 + \tau_{j2}(x)^2} \geq 1$, hence $\max \{|\tau_{j1}(x)|_{\infty}, |\tau_{j2}(x)|_{\infty}\} \geq \frac{1}{\sqrt{2}}$. Therefore,
\begin{equation}
\label{c3}
\max \{ |\sigma_1(x)|,...,|\sigma_{r_1}(x)|,|\tau_{11}(x)|,|\tau_{12}(x)|,...,|\tau_{r_2 1}(x)|,|\tau_{r_2 2}(x)| \} \geq \frac{1}{\sqrt{2}},
\end{equation}
in other words the maximum of the Euclidean absolute values of all conjugates of an algebraic integer is at least $\frac{1}{\sqrt{2}}$.

Finally, recall that 
\begin{equation}
\label{n2}
\Delta := |\det(\Lambda)| = \frac{|\D_K|^{1/2}}{2^{r_2}},
\end{equation}
which follows immediately from Lemma 2 on p. 115 of \cite{lang}. We are now ready to apply Lemma \ref{2.4.1}. By Corollary 1 on p. 13 of \cite{cass:geom}, we can select a basis for $\Lambda$ so that the basis matrix is upper triangular, all of its nonzero entries are positive, and the maximum entry of each row occurs on the diagonal. By (\ref{c3}) each of these maximum values is at least $\frac{1}{\sqrt{2}}$, so the lattice $\Lambda$ satisfies the conditions of Lemma \ref{2.4.1} with $c=\frac{1}{\sqrt{2}}$, $n=d$, and $\Delta$ as in (\ref{n2}). Therefore, if we take $R \geq \left( 2^{r_1} |\D_K| \right)^{1/2}$, then by (\ref{c2}) combined with Lemma \ref{2.4.1}
\begin{eqnarray}
\label{n3}
|S_R(K)| \geq |\Lambda \cap C_{R/\sqrt{2}}^d| & \geq & \left( \frac{R}{2^{\frac{r_1-2}{2}} |\D_K|^{1/2}} - 1 \right) (2R - 1)^{d-1} \nonumber \\
& > & \left( 2^{r_1} |\D_K| \right)^{-1/2} R^d,
\end{eqnarray}
which proves the lower bound of (\ref{count_mech}). Also
\begin{eqnarray}
\label{n4}
|S_R(K)| \leq |\Lambda \cap C_{R}^d| & \leq & \left( \frac{R}{2^{\frac{r_1-3}{2}} |\D_K|^{1/2}} + 1 \right) \left(2\sqrt{2}R + 1\right)^{d-1} \nonumber \\
& < & 2^{2d+1/2} \left( 2^{r_1} |\D_K| \right)^{-1/2} R^d,
\end{eqnarray}
which proves the upper bound of (\ref{count_mech}).
\endproof

\noindent
We can now easily derive Corollary \ref{lower_bnd}.

\begin{proof}
[Proof of Corollary \ref{lower_bnd}]
Notice that $R  \geq \left( 2^{r_1} |D_K| \right)^{1/2}>1$, so if $x \in S_R(K)$, then
$$h(x) = \prod_{v \in M(K)} \max \{ 1, |x|_v \}^{d_v/d} \leq \prod_{v \in M(K)} R^{d_v/d} = R,$$
hence $S_R(K) \subseteq \{ x \in O_K\ :\ h(x) \leq R \}$. The statement of the corollary now follows from Lemma \ref{count}.
\end{proof}
\bigskip

\section{A counting mechanism: function field case}
\label{funct_count}

Here we produce a counting estimate analogous to Lemma \ref{count} over a function field with a finite field of constants. First we recall a lemma (Theorems 4.2 and 4.3 of \cite{me:classical}) which we will need here.

\begin{lem} [\cite{me:classical}] \label{2.4.3} Suppose that $\Lambda \subseteq \zed^n$ is a lattice of rank $n-l$, where $1 \leq l \leq n-1$. Let $\Delta$ be the maximum of absolute values of Grassmann coordinates of $\Lambda$. Then for every $R$ that is a positive integer multiple of $(n-l) \Delta$, we have
\begin{equation}
\label{notfull:rank}
\frac{(2R)^{n-l}}{(n-l)^{n-l} \Delta} \leq |\Lambda \cap C_R^n| \leq \left( \frac{2R}{\Delta}+1 \right)(2R+1)^{n-l-1},
\end{equation}
where $C_R^n$ is as in (\ref{cube}). The upper bound of (\ref{notfull:rank}) holds for $R$ that is not an  integer multiple of $(n-l) \Delta$ as well.
\end{lem}

\noindent
The following is a construction of function field lattices (FML) as on pages 578--583 of \cite{tsfasman}. Let $K$ be a function field over a finite field $\ff_q$ for a prime power $q$, then there exists a curve $Y$ over $\ff_q$ such that $K = \ff_q(Y)$ is the field of rational functions on $Y$. Let the set of points of $Y$ over $\ff_q$ be
$$Y(\ff_q) = \{ P_1,\dots,P_{n(K)} \},$$
where $n(K) = |Y(\ff_q)|$, and let $\M_Y = \{ v_1,\dots,v_{n(K)} \} \subset M(K)$ be a subset of places of $K$ corresponding to these points. In other words, for every $f \in K$ and for each $1 \leq i \leq n(K)$, we have $|f|_{v_i} = e^{-\ord_{v_i} (f)}$, where
\[ \ord_{v_i} (f)= \left\{ \begin{array}{ll}
k & \mbox{if $f$ has a zero of multiplicity $k$ at $P_i$} \\
-k & \mbox{if $f$ has a pole of multiplicity $k$ at $P_i$} \\
0 & \mbox{otherwise}.
\end{array}
\right. \]
Let
$$O_K(Y) = \left\{ f \in K^* : \ord_v (f) = 0\ \forall\ v \in M(K) \setminus \M_Y \right\}$$
be the ring of rational functions from $K$ with zeros and poles only at the places in $\M_Y$. Then for each $f \in O_K(Y)$
$$\sum_{v \in \M_Y} \ord_v (f) = 0,$$
since $f$ defines a principal divisor. Define
$$\H_{n(K)} = \left\{ \bx \in \real^{n(K)} : \sum_{i=1}^{n(K)} x_i = 0 \right\},$$
so $\H_{n(K)}$ is an $(n(K)-1)$-dimensional subspace of $\real^{n(K)}$. We now have a natural embedding $\varphi_Y : O_K(Y) \rightarrow \zed^{n(K)} \cap \H_{n(K)}$ given by
$$\varphi_Y(f) = ( \ord_{v_1}(f),\dots,\ord_{v_{n(K)}}(f)).$$
Then $\ker(\varphi_Y) = \ff_q^*$; also, by Theorem 5.4.9 on p. 579 of \cite{tsfasman}, $\Lambda_Y := \varphi(O_K(Y))$ is a lattice of full rank in $\H_{n(K)}$, hence a sublattice of $\zed^{n(K)}$ of rank $n(K)-1$, and
\begin{equation}
\label{dltn}
\sqrt{n(K)} \leq \det \Lambda_Y \leq \sqrt{n(K)}\ h_K \leq \sqrt{n(K)} \left( \frac{(g(K)-1)(q+1)+n(K)}{g(K)} \right)^{g(K)},
\end{equation}
where $h_K$ is the class number of $K$, and $g(K)$ is the genus of $Y$, and hence of $K$. If $g(K)=0$, the upper bound of (\ref{dltn}) becomes simply $\sqrt{n(K)}$, thus enforcing equality throughout ($h(K)=1$ in this case). For a positive real number $R$ define
\begin{equation}
\label{SRK_f}
S_R(K) = \left\{ f \in O_K(Y)\ :\ \ord_v(f) \leq R\ \forall\ v \in \M_Y \right\},
\end{equation}
then
\begin{equation}
\label{SRK_f1}
|S_R(K)| = |\Lambda_Y \cap C_R^{n(K)}| + |\ker(\varphi_Y)| = |\Lambda_Y \cap C_R^{n(K)}| + q-1,
\end{equation}
and we have the following estimate.

\begin{lem} \label{count_f} For every real number $R \geq (n(K)-1) \sqrt{n(K)}\ h_K$ ,
\begin{eqnarray}
\label{count_mech_f}
& & \frac{2^{n(K)-1}}{\sqrt{n(K)}\ h_K } \left( \frac{R}{n(K)-1} - \sqrt{n(K)}\ h_K \right)^{n(K)-1} + q-1 \nonumber \\
& & \leq |S_R(K)| \leq (2R+1)^{n(K)-1} + q-1.
\end{eqnarray}
\end{lem}

\proof
By (\ref{SRK_f1}), we need to estimate $|\Lambda_Y \cap C_R^{n(K)}|$. Let $\Delta_Y$ be the maximum of absolute values of Grassmann coordinates of $\Lambda_Y$. By Cauchy-Binet formula
\begin{equation}
\label{cmf1}
\Delta_Y \leq \det \Lambda_Y \leq \sqrt{n(K)} \Delta_Y.
\end{equation}
Let $R_1 = \left[ \frac{R}{(n(K)-1) \Delta_Y} \right] (n(K)-1) \Delta_Y$, where $[\ ]$ denotes the integer part function, then by combining Lemma \ref{2.4.3} with (\ref{cmf1}), we have
\begin{eqnarray}
\label{cmf2}
|\Lambda_Y \cap C_R^{n(K)}| & \geq & |\Lambda_Y \cap C_{R_1}^{n(K)}| \geq \frac{(2R_1)^{n(K)-1}}{(n(K)-1)^{n(K)-1} \Delta_Y} \nonumber \\
& = & 2^{n(K)-1} \Delta_Y^{n(K)-2} \left[ \frac{R}{(n(K)-1) \Delta_Y} \right]^{n(K)-1} \nonumber \\
& \geq & \frac{2^{n(K)-1}}{\Delta_Y} \left( \frac{R}{n(K)-1} - \Delta_Y \right)^{n(K)-1} \nonumber \\
& \geq & \frac{2^{n(K)-1}}{\det \Lambda_Y} \left( \frac{R}{n(K)-1} - \det \Lambda_Y \right)^{n(K)-1}.
\end{eqnarray}
The lower bound of (\ref{count_mech_f}) follows by combining (\ref{cmf2}) with (\ref{dltn}) and (\ref{SRK_f1}). The upper bound also follows readily by combining Lemma \ref{2.4.3} with (\ref{SRK_f1}), (\ref{cmf1}) and~(\ref{dltn}).
\endproof
\bigskip

\section{Proof of Theorem \ref{general_main}}
\label{proof}

In this section we prove our main result. All the notation is as in section~\ref{notation} and in the statement of Theorem \ref{general_main}. Let $K$ be an admissible field, let $V \subseteq K^N$ be an $L$-dimensional vector space, and let $\bv_1,\dots,\bv_L$ be the basis for $V$ guaranteed by Theorems \ref{siegel} and \ref{siegel_fi} (inequalities (\ref{siegel_bound1}) and (\ref{siegel_bound2})). We will start by proving the theorem for the case of just one polynomial $P(X_1,\dots,X_N)$ of degree $M$, in other words first suppose $\Z_K = Z_K(P)$. Assume that $P$ is not identically zero on $V$, so $V \nsubseteq Z_K(P)$. We will prove the existence of a point $\bx \in V \setminus Z_K(P)$ satisfying~(\ref{gen_bnd_1}).

Let $S_1$ be a finite subset of $K$ such that $|S_1| > M$, and let $S=S_1^L$.  Then, by Theorem \ref{combin_null_1}, there exists $\bxi \in S$ such that $P(\bv(\bxi)) \neq 0$, where
\begin{equation}
\label{bvxi}
\bv(\bxi) = \sum_{i=1}^L \xi_i \bv_i \in V.
\end{equation}
By Lemma \ref{sum_height} combined with Theorems \ref{siegel} and \ref{siegel_fi}
\begin{equation}
\label{zero_bnd1}
H(\bv(\bxi)) \leq h(\bv(\bxi)) \leq L^{\delta} h(\bxi) \prod_{i=1}^L h(\bv_i) \leq L^{\delta} \E_K(L)^{1-\delta} C_K(L) h(\bxi) \H(V).
\end{equation}
We now want to select the set $S_1$ in a way that would minimize $h(\bxi)$; this choice will depend on the nature of the field $K$. We will show that the upper bound on $h(\bxi)$ is precisely the constant $A_K(L,M)$ as in (\ref{AKLM}). Then we can take $\bx$ in the statement of Theorem \ref{general_main} to be~$\bv(\bxi)$.
\smallskip

First assume that $K$ is a number field with $\omega_K \leq M$. Then take 
$$R = \left( 2^{r_1} |\D_K| \right)^{1/2d} M^{1/d},$$
and let $S_1 = S_R(K)$, where $S_R(K)$ is as in (\ref{SRK}). By Lemma \ref{count}
$$|S_R(K)| > \left( 2^{r_1} |\D_K| \right)^{-1/2} R^d = M,$$
therefore $|S_R(K)| \geq M+1$. We now can estimate $h(\bxi)$. Since $\bxi \in S = S_R(K)^L$,
$$H_v(\bxi) \leq 1\ \forall\ v \nmid \infty,\ H_v(\bxi) \leq R^{d_v}\ \forall\ v | \infty,$$
therefore, since $R>1$
\begin{equation}
\label{zero_bnd2}
h(\bxi) \leq R =  \left( 2^{r_1} |\D_K| \right)^{1/2d} M^{1/d}.
\end{equation}
Combining (\ref{zero_bnd1}) with (\ref{zero_bnd2}) produces (\ref{gen_bnd_1}).
\smallskip

\begin{rem} \label{M_optimal} Notice that in our choice of $R = \left( 2^{r_1} |\D_K| \right)^{1/2d} M^{1/d}$ in the argument above it is essential to take $M^{1/d}$: if we take a smaller power of $M$, then $|S_R(K)|$ can be smaller than $M+1$, in which case a polynomial $P_V$ could vanish identically on $S_R(K)^L$. Indeed, as is discussed in \cite{me:classical}, if $S_1 = \{\alpha_1,...,\alpha_M\} \subset K$ and
$$P(X_1,...,X_N) = \sum_{i=1}^N \prod_{j=1}^M (X_i - \alpha_j),$$
then for each $\bx \in S_1^N$ we have $P(\bx) = 0$.
\end{rem}
\smallskip

Next suppose that $K$ is an admissible function field of finite type $q \leq M$. Let $Y$ be the smooth projective curve so that $K = \ff_q(Y)$, as in section~\ref{funct_count}. Then take $R=R_K(M)$ as in (\ref{RK}), and let $S_1 = S_R(K)$, where $S_R(K)$ is as in (\ref{SRK_f}). By Lemma \ref{count_f}
$$|S_R(K)| \geq \frac{2^{n(K)-1}}{\sqrt{n(K)}\ h_K } \left( \frac{R}{n(K)-1} - \sqrt{n(K)}\ h_K \right)^{n(K)-1} + q-1 = M+1.$$
We now can estimate $h(\bxi)$. Since $\bxi \in S = S_R(K)^L$,
$$H_v(\bxi) = 1\ \forall\ v \notin \M_Y,\ H_v(\bxi) \leq e^{Rd_v}\ \forall\ v \in \M_Y,$$
therefore
\begin{equation}
\label{zero_bnd3}
h(\bxi) \leq e^{R_K(M)}.
\end{equation}
Combining (\ref{zero_bnd1}) with (\ref{zero_bnd3}) produces (\ref{gen_bnd_1}).
\smallskip

\begin{rem} \label{Thunder_JNT} Another way of selecting the set $S_1$ in case of a function field $K$ of finite type $q \leq M$ is by employing bounds on the number of elements of $K$ of bounded height as in \cite{thunder1}. Specifically, Corollary 1 of \cite{thunder1} with $n=2$ and $m=R$ implies that there exists a constant $T(K)$ such that the number of elements $f \in K$ with height $h(f) \leq e^R$ is $> T(K) q^{2R}$. If we pick 
\begin{equation}
\label{R_thunder}
R = \frac{1}{2 \log q} \log \left( \frac{M}{T(K)} \right),
\end{equation}
then the set
$$S_1 = \{ f \in K : h(f) \leq e^R \}$$
will have cardinality $|S_1| \geq M+1$. Taking $S=S_1^L$, and letting $\bxi \in S$ guarantees that
$$h(\bxi) \leq \prod_{i=1}^L h(\xi_i) \leq e^{LR},$$
and so we can take $A_K(L,M) = e^{LR}$ with $R$ as in (\ref{R_thunder}). It should be remarked however that Thunder's estimate in Corollary 1 of \cite{thunder1} is asymptotic, and so an explicit value for the constant $T(K)$ is not specified.
\end{rem}
\smallskip

Now suppose that $K$ is any other admissible field except for those discussed above (i.e. $K$ is either a number field with $\omega_K > M$, an admissible function field of finite type $q > M$ or of infinite type, or $K=\qbar$). Then $K$ contains a set $S_1$ of cardinality at least $M+1$ such that for every $\xi \in S_1$ and every $v \in M(K)$, $|\xi|_v = 1$.  Let $S=S_1^L$, and notice that for each $\bxi \in S$, $h(\bxi)=1$. Combining this observation with (\ref{zero_bnd1}) produces (\ref{gen_bnd_1}).
\smallskip

We have so far proved Theorem \ref{general_main} for the case when $\Z_K$ is just a hypersurface defined over $K$. We can now extend our argument to any finite union of varieties $\Z_K$ as in the statement of Theorem \ref{general_main}. Since $V \nsubseteq \Z_K$, $V \nsubseteq Z_K(P_{i1},\dots,P_{ik_i})$ for all $1 \leq i \leq J$, and so for each $i$ at least one of the polynomials $P_{i1},\dots,P_{ik_i}$ is not identically zero on $V$, say it is $P_{ij_i}$ for some $1 \leq j_i \leq k_i$. Clearly for each $1 \leq i \leq J$, $Z_K(P_{i1},\dots,P_{ik_i}) \subseteq Z_K(P_{ij_i})$, and  $\deg(P_{ij_i}) = m_{ij_i} \leq M_i$. Define 
$$P(X_1,\dots,X_N) = \prod_{i=1}^J P_{ij_i}(X_1,\dots,X_N),$$
so that $V \nsubseteq Z_K(P)$ while $\Z_K \subseteq Z_K(P)$. Then it is sufficient to construct a point of bounded height $\bx \in V \setminus Z_K(P)$. Now notice that $\deg(P) = \sum_{i=1}^J m_{ij_i} \leq M$ and apply our argument above for the case of just one polynomial. This completes the proof of the theorem.
\bigskip

\section{Twisted height}
\label{twisted}

In this section we remark that all the results of this paper extend to bounds on {\it twisted height} of the point in question. Let us write $K_{\aaa}$ for the ring of adeles of $K$, and view $K$ as a subfield of $K_{\aaa}$ under the diagonal embedding (see \cite{weil} for details). Let $A \in GL_N(K_{\aaa})$ with local components $A_v \in GL_N(K_v)$. The corresponding twisted height on $K^N$ (as introduced by J. L. Thunder) is defined by
\begin{equation}
H_A(\bx) = \left( \prod_{v \in M(K)} H_v(A_v \bx) \right)^{1/d},
\end{equation}
for all $\bx \in K^N$. Given any finite extension $E/K$, $K_{\aaa}$ can be viewed as a subring of $E_{\aaa}$, and let us also write $A$ for the element of $GL_N(E_{\aaa})$ which coincides with $A$ on $K_{\aaa}^N$. The corresponding twisted height on $E^N$ extends the one on $K^N$, hence $H_A$ is a height on $\kbar$. Notice also that the usual height $H$ as defined above is simply $H_I$, where $I$ is the identity element of $GL_N(K_{\aaa})$ all of whose local components are given by $N \times N$ identity matrices. 

For each element $A \in  GL_N(K_{\aaa})$, the height $H_A$ is comparable to the canonical height $H$ by means of certain dilation constants that, roughly speaking, indicate by how much does a given automorphism $A$ of $K_{\aaa}^N$ "distort" the corresponding twisted height $H_A$ as compared to $H$. We will only need one of these constants. Let $A_v = (a^v_{ij})_{1 \leq i,j \leq N} \in  GL_N(K_v)$ be local components of $A$ for each $v \in M(K)$. Then for all but finitely many places $v \in M(K)$ the corresponding map $A_v$ is an isometry; in fact, let $M_A(K) \subset M(K)$ be the finite (possibly empty) subset of places $v$ at which $A_v$ is {\it not} an isometry. For each $v \notin M_A(K)$, define $\C_v(A) = 1$, and for each $v \in M_A(K)$, let
\begin{equation}
\label{aut_const_loc}
\C_v(A) = \sum_{i=1}^N \sum_{j=1}^N |a^v_{ij}|_v,
\end{equation}
and define
\begin{equation}
\label{aut_const}
\C(A) = \prod_{v \in M(K)} \C_v^{d_v/d},
\end{equation}
which is a product of only a finite number of non-trivial terms. Clearly, in the case when $A = I$ is the identity element of $GL_N(K_{\aaa})$, $ \C(A) = 1$. Then Proposition 4.1 of \cite{absolute:siegel} states that
\begin{equation}
\label{comp1}
H_A(\bx) \leq \C(A) H(\bx),
\end{equation}
for all $\bx \in \qbar^N$. Now one can use (\ref{comp1}) to restate Theorem \ref{general_main} replacing $H(\bx)$ by $H_A(\bx)$ - the only change is the appearance of the dilation constant $\C(A)$ in the upper bound.
\bigskip

{\bf Acknowledgment.} I would like to thank Wai Kiu Chan for letting me use his nice idea that allowed to improve the result of Theorem \ref{siegel_fi}, and for his very useful comments on the subject of this paper. I would also like to thank Jeff Thunder for his helpful remarks, and for providing me with a preprint copy of his paper \cite{thunder1}. Finally, I would like to acknowledge the wonderful hospitality of Max-Planck-Institut f\"{u}r Mathematik in Bonn, Germany, where a large part of this work has been done. An earlier version of this paper also appears in MPIM preprint series.

\bibliographystyle{plain}  
\bibliography{null}        

\end{document}